\documentclass[11pt]{amsart}
\usepackage{amsmath,amssymb,amsthm,mathtools}
\usepackage{enumitem}
\usepackage{booktabs}
\usepackage{microtype}
\usepackage[hidelinks]{hyperref}

\newtheorem{theorem}{Theorem}[section]
\newtheorem{lemma}[theorem]{Lemma}
\newtheorem{proposition}[theorem]{Proposition}

\newtheorem{conjecture}[theorem]{Conjecture}
\theoremstyle{remark}
\newtheorem{remark}[theorem]{Remark}

\newcommand{\e}{\mathrm{e}}
\newcommand{\R}{\mathbb R}
\newcommand{\Q}{\mathbb Q}
\newcommand{\Z}{\mathbb Z}
\newcommand{\Hh}{\mathbb H}
\newcommand{\Res}{\operatorname*{Res}}
\newcommand{\vol}{\operatorname{vol}}

\newcommand{\wt}{\widetilde}

\title[Sign changes in the congruent number problem]
{Signs of Square-Free Fourier Coefficients of half-integral weight cusp forms and the Congruent Number Problem}
\author{Tao Wei}
\address{Department of Mathematics, Jinling Institute of Technology, Nanjing 211169, China}
\email{weitao@jit.edu.cn}
\author{Xuejun Guo}
\address{Department of Mathematics, Nanjing University, Nanjing 210093, China}
\email{guoxj@nju.edu.cn}

\thanks{The authors are supported by the National Natural Science Foundation of China (12231009).}

\begin{document}

\begin{abstract}
	We study the sign distribution of square-free Fourier coefficients of
	half-integral weight cusp forms. We prove that, for half-integral weight
	cusp forms \(f\) satisfying the eigenform conditions and having a nonzero cuspidal Shimura lift, the numbers of positive and negative square-free
	Fourier coefficients up to \(X\) are both
	\(\gg_{f,\varepsilon}X^{4/7-\varepsilon}\). 
	
	As an application to the congruent-number problem, we consider the weight \(3/2\) cusp form whose Shimura lift is the weight \(2\) newform attached to $ E:y^2=x^3-x$. We prove that
	each each sign occurs among its square-free Fourier coefficients at  \(\gg_{\varepsilon}X^{4/7-\varepsilon}\) odd square-free
	integers \(t\leq X\). In particular, the square-free Fourier
	coefficients change sign infinitely often.
\end{abstract}

\maketitle

\section{Introduction}
A positive integer $t$ is called a congruent number if it is the
area of a right triangle with rational side lengths. For square-free $t$,
this is equivalent to the condition that the elliptic curve
\[
E_t:y^2=x^3-t^2x
\]
has positive Mordell--Weil rank over $\Q$. The curves $E_t$ are the
quadratic twists of $E:y^2=x^3-x$, and the problem of determining which
$t$ are congruent is a classical Diophantine problem associated with this
family.

For a ternary quadratic form \(Q\), let
\[
r_Q(n)=\#\{(x,y,z)\in\Z^3:Q(x,y,z)=n\}
\]
denote the number of representations of \(n\) by \(Q\). For odd square-free
\(t\), Tunnell's criterion is expressed in terms of the quadratic forms
\[
Q_T(x,y,z)=x^2+2y^2+8z^2,
\qquad
Q_1(x,y,z)=x^2+2y^2+32z^2.
\]
If $t$ is congruent, then
\begin{equation}\label{eq:tunnell-odd}
r_{Q_T}(t)=2r_{Q_1}(t),
\end{equation}
and the converse follows from the weak Birch--Swinnerton-Dyer conjecture
\cite{Tunnell}. Tunnell's proof relates the arithmetic function
$2r_{Q_1}(t)-r_{Q_T}(t)$ on odd indices to the Fourier coefficients of a
weight $3/2$ Hecke cusp eigenform whose Shimura lift is the weight $2$
newform attached to $E$. 

Qin obtained an alternative criterion in which the relevant weight $3/2$
form is itself a difference of two ternary theta series \cite{Qin}. In
addition to $Q_1$, consider
\[
Q_2(x,y,z)=2x^2+4y^2+9z^2-4yz.
\]

Put
\[
\Delta(n)=r_{Q_1}(n)-r_{Q_2}(n)
\]
and
\[
\Psi(z)=\frac12\bigl(\Theta_{Q_1}(z)-\Theta_{Q_2}(z)\bigr)
=\sum_{n\geq1}\frac{\Delta(n)}2q^n.
\]
Qin proved that $\Psi$ is an eigenform for $T(p^2)$ at every odd prime
$p$, and identified its Shimura lift with the weight $2$ newform attached
to $E$. A special case of Waldspurger's formula (see \cite{Qin}) gives that,
for every positive odd square-free integer \(t\),
\begin{equation}\label{eq:central-value}
	L(E_t,1)=\frac{\Delta(t)^2}{16\sqrt{t}}
	\int_1^\infty\frac{dx}{\sqrt{x^3-x}}.
\end{equation}

Formula~\eqref{eq:central-value} determines the absolute value of
$\Delta(t)$, but not its sign. Rodriguez--Villegas pointed out that the Fourier coefficients appearing in Waldspurger-type formulas are not necessarily of one sign, and no general arithmetic interpretation of their signs is known \cite{RodriguezVillegas}. Here the sign of $\Delta(t)$ distinguishes
\[
r_{Q_1}(t)>r_{Q_2}(t)
\qquad\text{from}\qquad
r_{Q_1}(t)<r_{Q_2}(t).
\]

In this paper, we show that each inequality holds for $\gg_\varepsilon X^{4/7-\varepsilon}$ odd square-free integers $t\leq X$. In particular, $\Delta(t)$ changes sign infinitely often. Every integer counted by either inequality is non-congruent. Indeed, the theorem of Coates and Wiles implies, since $E_t$ has complex multiplication, that $L(E_t,1)\ne0$ forces $E_t(\Q)$ to have Mordell--Weil rank zero \cite{CoatesWiles}.

The problem is thus a special case of the sign problem for square-free
Fourier coefficients of half-integral weight cusp forms. The signs of Fourier coefficients of half-integral weight forms have
been studied from several viewpoints. For a fixed square-free $t$ with
nonzero coefficient, Bruinier and Kohnen proved infinitely many sign
changes in prime-power subsequences inside the square class $tn^2$
\cite{BruinierKohnen}. Hulse, K\i ral, Kuan and Lim proved infinite sign
changes of square-free coefficients for level $4$ Hecke eigenforms by
continuing a square-free coefficient Dirichlet series to
$\Re(s)>3/4$ \cite{HKKL}. Lau, Royer and Wu combined this continuation
with Rankin--Selberg second moments to obtain polynomial lower bounds for
the numbers of positive and negative coefficients, as well as a separate
short-interval estimate for the number of sign changes, again at level
$4$ and in higher weight \cite{LRW}.

Jiang, Lau, L\"u, Royer and Wu later moved the continuation line to
$\Re(s)>1/2$ at level $4$ by exploiting additional cancellation in
additive twists \cite{JLRW}. Stronger results, including
positive-proportion results under the generalized Riemann hypothesis,
were established by Lester and Radziwi\l\l{} in the level-$4$ plus space,
with extensions to forms of level $4N$ when $N$ is odd and square-free
\cite{LesterRadz}.

The main new issue in passing from level $4$ to a general fixed level
is the treatment of additive twists. For level \(4\), the proof in
\cite{HKKL} distinguishes three cases according as the cusp represented by
\(u/d\) is equivalent to \(\infty\), \(1/2\), or \(0\).
For a general level \(N\), the cusp class of \(u/d\) depends on the
common divisors of \(d\) and \(N\). We establish a functional equation
for the additive twist at \(u/d\) in terms of the Fourier expansion of
\(f\) at the corresponding cusp. 

Using these functional equations, we continue
\[
M_f(s)=\sum_{\substack{t\geq1\\t\ \mathrm{square\text{-}free}}}
\frac{a_f(t)}{t^s}
\]
holomorphically to $\Re(s)>3/4$ for every fixed \(N\) divisible by \(4\). Combined with the Rankin--Selberg second moment at level
$N$, the square-free transfer of \cite[Lemma~9]{LRW}, and the pointwise
bound for square-free coefficients, this yields a lower bound
for each sign.

Throughout the paper, we normalize the Fourier expansion by
\[
f(z)=\sum_{n\ge1}a_f(n)n^{k/4-1/2}\e(nz),\qquad \e(z)=\exp(2\pi iz),
\]
for a form of weight $k/2$, with $k$ odd. Define
\[
P_f^\pm(X)=\#\{t\le X:t\text{ square-free and }\pm a_f(t)>0\}.
\]

Let $U_{3/2}(N)$ denote the subspace of $S_{3/2}(\Gamma_0(N))$ spanned
by unary theta series. For the trivial character, this is the space denoted
by $S_0(N,\mathbf 1)$ in \cite{PurkaitDecomposition}.

We now state our main sign result for cusp forms of level divisible by $4$.

\begin{theorem}\label{thm:general}
Let $N$ be a positive integer with $4\mid N$, $k\ge 3$ be odd. Let
$f\in S_{k/2}(\Gamma_0(N))$
have real normalized Fourier coefficients. Suppose that
\begin{enumerate}[label=\textup{(\arabic*)}]
\item for every prime $p$, one has
$f|T(p^2)=\lambda_f(p)f$ for some scalar $\lambda_f(p)$, where
$T(p^2)$ denotes $U(p^2)$ when $p\mid N$;
\item $f$ has a nonzero cuspidal Shimura lift;
\item when $k=3$, the form $f$ is orthogonal to the unary theta subspace
$U_{3/2}(N)$.
\end{enumerate}
Then, for every $\varepsilon>0$,
\[
P_f^+(X)\gg_{f,\varepsilon}X^{4/7-\varepsilon},\qquad
P_f^-(X)\gg_{f,\varepsilon}X^{4/7-\varepsilon}.
\]
In particular, both signs occur infinitely often among the nonzero
square-free coefficients of $f$, and these coefficients, ordered by their
indices, change sign infinitely often.
\end{theorem}

By \cite{Qin}, $\Psi$ is a Hecke eigenform at every odd prime and its
Shimura lift is the weight $2$ newform attached to $E$.
Section~\ref{sec:application} verifies the condition of Theorem~\ref{thm:general} and gives the following.

\begin{theorem}\label{thm:quadratic}
For every $\varepsilon>0$,
\[
\#\{t\le X:t\text{ odd and square-free},\ r_{Q_1}(t)>r_{Q_2}(t)\}
\gg_\varepsilon X^{4/7-\varepsilon},
\]
and
\[
\#\{t\le X:t\text{ odd and square-free},\ r_{Q_1}(t)<r_{Q_2}(t)\}
\gg_\varepsilon X^{4/7-\varepsilon}.
\]
In particular, $\Delta(t)$ changes sign infinitely often as $t$ ranges over
odd square-free positive integers.
\end{theorem}

Numerical calculations in Section~\ref{sec:application} suggest that,
among odd square-free integers $t$ with $\Delta(t)\ne0$, the two signs are
asymptotically equidistributed; see Table~\ref{tab:numerical-signs} and
Conjecture~\ref{conj:sign-equidistribution}.

The paper is organized as follows. In Section \ref{sec:functional} we prove functional equations for additive twists at arbitrary cusps. In Section \ref{sec:squarefree-series} we continue the square-free coefficient Dirichlet series to $\Re(s)>3/4$. In Section \ref{sec:second-moment} we establish the general-level Rankin--Selberg second moment, transfer it to square-free indices and prove Theorem \ref{thm:general}. In Section \ref{sec:application} we apply the general theorem to the theta-series difference associated with $Q_1$ and $Q_2$ and derive the corresponding result for non-congruent numbers.

\section{Functional equations at arbitrary cusps}\label{sec:functional}
The functional equation for an additive twist is obtained by moving the
rational point at which the twist is taken to a cusp of $\Gamma_0(N)$.
 Assume that
$4\mid N$, that $k$ is odd, and that $f\in S_{k/2}(\Gamma_0(N))$. For
$\gamma=\begin{pmatrix}a&b\\c&d\end{pmatrix}\in\Gamma_0(N)$, the factor of
automorphy is
\[
j(\gamma,z)=\varepsilon_d^{-1}\left(\frac cd\right)(cz+d)^{1/2},
\]
where $\varepsilon_d=1$ for $d\equiv1\pmod4$ and $\varepsilon_d=i$ for
$d\equiv3\pmod4$, and $\left(\frac cd\right)$ denotes Shimura's extension
of the Jacobi symbol. If $\xi_\gamma=(\gamma,j(\gamma,z))$ is the
corresponding metaplectic element, then
\[
f|[\xi_\gamma]_{k/2}(z)=j(\gamma,z)^{-k}f(\gamma z)
\]
in the notation of \cite{Shimura}. In particular,
$f|[\xi_\gamma]_{k/2}=f$ for $\gamma\in\Gamma_0(N)$.

For every cusp $\mathfrak a$ of $\Gamma_0(N)$, fix a scaling matrix
$\sigma_{\mathfrak a}\in\mathrm{SL}_2(\Z)$ satisfying
$\sigma_{\mathfrak a}\infty=\mathfrak a$, together with a lift
$\wt\sigma_{\mathfrak a}$ to the metaplectic cover, and set
$f_{\mathfrak a}=f|[\wt\sigma_{\mathfrak a}]_{k/2}$. If
$w_{\mathfrak a}$ is the cusp width, then
\[
f_{\mathfrak a}(z)=\sum_{n+\kappa_{\mathfrak a}>0}b_{\mathfrak a}(n)
\e\left(\frac{(n+\kappa_{\mathfrak a})z}{w_{\mathfrak a}}\right)
\]
for some multiplier shift $0\leq\kappa_{\mathfrak a}<1$.

Fix a lift of $S_0=\begin{pmatrix}0&-1\\1&0\end{pmatrix}$ and set
$f_0=f|[\wt S_0]_{k/2}$. For $q\in\Q$, define
\[
\Lambda(f,q,s)=\int_0^\infty f(iy+q)y^{s+k/4-1/2}\,\frac{dy}{y}.
\]
Since $f$ is cuspidal at every cusp, this integral is entire in $s$. We
write $s'=s+k/4-1/2$.

\begin{lemma}\label{lem:metaplectic-constant}
Let $A,B\in\mathrm{SL}_2(\R)$, and choose lifts $\widetilde A$,
$\widetilde B$ and $\widetilde{AB}$ to the metaplectic cover. The
automorphy factor of $\widetilde A\widetilde B$ differs from that of
$\widetilde{AB}$ by a constant of absolute value one. Consequently,
changing a lift in a slash operator changes the resulting function only
by such a constant.
\end{lemma}

\begin{proof}
The two automorphy factors are holomorphic square roots of the same
nowhere-vanishing holomorphic function on $\Hh$. Their quotient has square
one and is therefore constant, since $\Hh$ is connected. A different
choice of lift differs by an element of the center of the metaplectic
cover, which acts by a constant of absolute value one.
\end{proof}

At level $4$, Hulse--K\i ral--Kuan--Lim divide the denominator $d$ into the three cases $4\mid d$, $2\parallel d$ and $2\nmid d$, corresponding to the three cusps $\infty$, $1/2$ and $0$ \cite[Lemma 4.3]{HKKL}. The following lemma replaces that case-by-case decomposition by one valid at an arbitrary level divisible by $4$.

\begin{lemma}\label{lem:additive-fe}
Let $u/d\in\Q$, where $d>0$ and $(u,d)=1$. Put
\[
d_0=(d,N),\qquad d=d_0d_1,
\]
and
\[
N_0=\prod_{\substack{p^\nu\parallel N\\p\mid d}}p^\nu,
\qquad N=N_0N_1.
\]
Since $(d,N_1u)=1$, one may choose integers $v,e$, with odd $v>0$, such that
\[
de-N_1uv=1.
\]
\begin{enumerate}[label=\textup{(\arabic*)}]
\item If $(d,N)=1$, then there is a constant $\omega_0(u,d)$, depending on the chosen lift of $S_0$, such that $|\omega_0(u,d)|=1$ and
\[
\Lambda\left(f,\frac ud,s\right)
=\omega_0(u,d)d^{1-2s}
\Lambda\left(f_0,\frac{N_1v}{d},1-s\right).
\]
\item Suppose that $(d,N)\ne1$. Then there exist integers $\alpha,\beta$, with $\beta\ne0$, such that
\[
\alpha N_1-\beta d_1=1,
\qquad
\beta v\equiv\alpha d\pmod{N_0}.
\]
Set
\[
\sigma_\beta=
\begin{pmatrix}1&-d_1\\-\beta&N_1\alpha\end{pmatrix},
\qquad
f_{-1/\beta}=f|[\wt\sigma_\beta]_{k/2},
\]
where a lift $\wt\sigma_\beta$ is fixed. Then $\sigma_\beta\in\mathrm{SL}_2(\Z)$ and $\sigma_\beta\infty=-1/\beta$. Moreover, there is a constant $\omega_\beta$, depending on the chosen lift, such that $|\omega_\beta|=1$ and
\[
\Lambda\left(f,\frac ud,s\right)
=\omega_\beta d^{1-2s}
\Lambda\left(f_{-1/\beta},\frac{N_1v}{d},1-s\right).
\]
\end{enumerate}
\end{lemma}

\begin{proof}
We first justify the choice of $v$ and $e$. By the definitions of $N_0$ and
$N_1$, one has $(d,N_1)=1$, and hence $(d,N_1u)=1$. Thus the equation
\[
de-N_1uv=1
\]
has integral solutions. If $d$ is even, then $N_1u$ is odd and the equation
itself forces $v$ to be odd. If $d$ is odd, the general solution
\[
e=e_0+N_1ut,\qquad v=v_0+dt
\]
allows us to choose the parity of $v$. Adding a sufficiently large multiple
of $2d$ if necessary, we may also arrange that $v>0$.

Assume first that $(d,N)=1$. Then $N_1=N$, and
\[
\gamma=\begin{pmatrix}e&u\\Nv&d\end{pmatrix}\in\Gamma_0(N).
\]
Put
\[
w=\frac{Nv}{d}+iy.
\]
A direct calculation, using $de-Nuv=1$, gives
\[
\gamma S_0w
=\frac{u w-e}{d w-Nv}
=\frac ud+\frac{i}{d^2y}.
\tag{2.1}
\]
After the change of variables $y\mapsto 1/(d^2y)$ in the defining Mellin
integral, it follows that
\[
\Lambda\left(f,\frac ud,s\right)
=d^{-2s'}\int_0^\infty
f\left(\gamma S_0\left(\frac{Nv}{d}+iy\right)\right)
y^{-s'}\,\frac{dy}{y}.
\tag{2.2}
\]
The product of the automorphy factors attached to the chosen lifts of
$\gamma$ and $S_0$ and the automorphy factor of a chosen lift of
$\gamma S_0$ are holomorphic square roots of the same linear function
$d w-Nv$. Lemma~\ref{lem:metaplectic-constant} shows that their quotient
is constant. Since $f|[\xi_\gamma]_{k/2}=f$, there is therefore a constant
$\omega_0(u,d)$ with $|\omega_0(u,d)|=1$ such that
\[
f(\gamma S_0w)
=\omega_0(u,d)(d w-Nv)^{k/2}f_0(w).
\]
On the line $w=Nv/d+iy$, one has $d w-Nv=idy$. Substituting this into
(2.2), and absorbing the fixed power of $i$ into $\omega_0(u,d)$, gives
\[
\Lambda\left(f,\frac ud,s\right)
=\omega_0(u,d)d^{-2s'+k/2}
\Lambda\left(f_0,\frac{Nv}{d},1-s\right).
\]
Since $-2s'+k/2=1-2s$, this proves part~\textup{(1)}.

Now suppose that $(d,N)\ne1$. We first verify the existence of $\alpha$
and $\beta$. Since $(d_1,N_1)=1$, choose one solution $(\alpha_0,\beta_0)$
of
\[
\alpha_0N_1-\beta_0d_1=1.
\]
All solutions are of the form
\[
\alpha=\alpha_0+d_1t,
\qquad
\beta=\beta_0+N_1t,
\qquad t\in\Z.
\]
The remaining congruence is equivalent to
\[
(N_1v-dd_1)t\equiv \alpha_0d-\beta_0v\pmod{N_0}.
\tag{2.3}
\]
We claim that
\[
(N_1v-dd_1,N_0)=1.
\tag{2.4}
\]
Indeed, let $p\mid N_0$. Then $p\mid d$ and $p\nmid N_1$. Reducing
$de-N_1uv=1$ modulo $p$ shows that $p\nmid v$. Hence
\[
N_1v-dd_1\equiv N_1v\not\equiv0\pmod p.
\]
This proves (2.4), so (2.3) has a solution. Replacing $t$ by $t+N_0m$ if
necessary, we may also arrange that $\beta\ne0$.

Define
\[
\gamma=
\begin{pmatrix}
e\beta-uN_1\alpha&e-ud_1\\
N_1(v\beta-\alpha d)&N_1v-dd_1
\end{pmatrix}.
\]
A direct multiplication gives
\[
\gamma\sigma_\beta=
\begin{pmatrix}-u&e\\-d&N_1v\end{pmatrix}.
\tag{2.5}
\]
Since $de-N_1uv=1$, both matrices in (2.5) have determinant one. Moreover,
\[
N\mid N_1(v\beta-\alpha d)
\]
by the defining congruence, so $\gamma\in\Gamma_0(N)$.

Put
\[
w=\frac{N_1v}{d}+iy.
\]
Using (2.5), we obtain the exact fractional-linear identity
\[
\gamma\sigma_\beta w
=\frac{-u w+e}{-d w+N_1v}
=\frac ud+\frac{i}{d^2y}.
\tag{2.6}
\]
After the same change of variables as above,
\[
\Lambda\left(f,\frac ud,s\right)
=d^{-2s'}\int_0^\infty
f\left(\gamma\sigma_\beta
\left(\frac{N_1v}{d}+iy\right)\right)
y^{-s'}\,\frac{dy}{y}.
\tag{2.7}
\]
The lower row of $\gamma\sigma_\beta$ is $(-d,N_1v)$. The automorphy
factor obtained from the chosen lifts of $\gamma$ and $\sigma_\beta$ and
the factor attached to a lift of $\gamma\sigma_\beta$ are holomorphic
square roots of the same linear function $-d w+N_1v$. By
Lemma~\ref{lem:metaplectic-constant}, their quotient is constant. Hence
there is a constant $\omega_\beta$, with $|\omega_\beta|=1$ and independent
of $w$, such that
\[
f(\gamma\sigma_\beta w)
=\omega_\beta(-d w+N_1v)^{k/2}f_{-1/\beta}(w).
\]
On the line $w=N_1v/d+iy$, the linear factor is $-idy$. Substitution into
(2.7), followed by absorbing the fixed power of $-i$ into $\omega_\beta$,
gives
\[
\Lambda\left(f,\frac ud,s\right)
=\omega_\beta d^{1-2s}
\Lambda\left(f_{-1/\beta},\frac{N_1v}{d},1-s\right).
\]
This proves part~\textup{(2)}.
\end{proof}

\begin{remark}\label{rem:structural-fe}
We relate the scaling matrices in Lemma~\ref{lem:additive-fe} to the fixed
matrices chosen at the beginning of the section. Suppose that
$\sigma_\beta\infty$ is equivalent to the fixed cusp representative
$\mathfrak a$. Then there are $\gamma\in\Gamma_0(N)$ and $m\in\Z$ such
that
\[
\gamma\sigma_\beta=\sigma_{\mathfrak a}(\pm T^m),
\qquad
T=\begin{pmatrix}1&1\\0&1\end{pmatrix}.
\]
Indeed, choose $\gamma$ so that
$\gamma\sigma_\beta\infty=\sigma_{\mathfrak a}\infty$; then
$\sigma_{\mathfrak a}^{-1}\gamma\sigma_\beta$ belongs to the stabilizer
of $\infty$ in $\mathrm{SL}_2(\Z)$ and hence equals $\pm T^m$. After
choosing lifts, Lemma~\ref{lem:metaplectic-constant} gives a constant
$\zeta$, $|\zeta|=1$, for which
\[
f_{-1/\beta}(z)=\zeta f_{\mathfrak a}(z+m),
\qquad
\Lambda(f_{-1/\beta},q,s)
 =\zeta\Lambda(f_{\mathfrak a},q+m,s).
\]
The same argument applies to the matrix used in part~\textup{(1)}. Thus,
for the analytic argument below, the functional equation always has the
form
\[
\Lambda\left(f,\frac ud,s\right)
=\omega(u,d)d^{1-2s}\Lambda(f_{\mathfrak a},q',1-s)
\]
for one of the finitely many fixed cusp expansions $f_{\mathfrak a}$,
some $q'\in\Q$, and $|\omega(u,d)|=1$.
\end{remark}

\section{The square-free coefficient Dirichlet series}\label{sec:squarefree-series}

Hulse, K{\i}ral, Kuan and Lim continue the square-free coefficient
Dirichlet series at level $4$ by treating separately the three cusps of
$\Gamma_0(4)$. For a general level $N$, Lemma~\ref{lem:additive-fe}
replaces this three-case analysis. Applied to the additive twists occurring
in $D_r(s)$, it gives the estimate on the left boundary of the interpolation
strip, while the estimate on the right boundary follows from absolute
convergence. The Phragm\'en--Lindel\"of principle then yields the decay in
$r$ needed to sum over the square divisors.

Throughout Sections~\ref{sec:squarefree-series} and
\ref{sec:second-moment}, $f$ is an eigenform for $T(p^2)$ at every prime
and has a nonzero cuspidal Shimura lift. Real Fourier coefficients and, in
weight $3/2$, orthogonality to the unary theta subspace are used only in
the proof of Theorem~\ref{thm:general}. We shall use the notation
\[
\widehat a_f(n)=a_f(n)n^{k/4-1/2},
\qquad
\chi_t(n)=\left(\frac{(-1)^{(k-1)/2}t}{n}\right)
\]
for square-free positive integers $t$, where the second symbol is the
Kronecker symbol.

\begin{lemma}\label{lem:square-index}
Let $t$ be square-free. For every $\varepsilon>0$,
\[
a_f(tn^2)\ll_{f,\varepsilon}|a_f(t)|n^\varepsilon.
\]
\end{lemma}

\begin{proof}
By Shimura's theorem in the form recorded in
\cite[Theorem~2(iv)]{PurkaitDecomposition}, the Hecke eigenvalues of $f$ determine a normalized integral-weight Hecke
eigenform
\[
F(z)=\sum_{m\ge1}A_F(m)\e(mz),\qquad A_F(1)=1.
\]
Since $f$ has a nonzero cuspidal Shimura lift, this form $F$ is cuspidal.
No coprimality condition is imposed on $n$ in
Lemma~\ref{lem:square-index}, so the Hecke relations at primes dividing
$N$ also enter the argument. Purkait proved that the Shimura
correspondence is compatible with $T(p^2)$ at every prime, including the
primes dividing $tN$; see \cite[Theorem~1.2]{PurkaitHecke}.

We first consider the case $a_f(t)=0$. The Fourier-coefficient formula
for $T(p^2)$ in \cite[Theorem~2.1]{PurkaitHecke} shows inductively that
\[
\widehat a_f(tn^2)=0
\]
for every $n\ge1$. We argue by induction on $\Omega(n)$, the number of prime factors of $n$
counted with multiplicity. Suppose that the assertion is known for all
integers with smaller $\Omega$-value, and write $n=pm$. The coefficient of
$q^{tm^2}$ in $T(p^2)f=\lambda_f(p)f$ expresses
$\widehat a_f(tp^2m^2)$ as a linear combination of
$\widehat a_f(tm^2)$ and, when $p\mid m$,
$\widehat a_f(tm^2/p^2)$. Both terms vanish by induction. Thus the lemma
is immediate when $a_f(t)=0$.

Suppose now that $a_f(t)\ne0$. By
\cite[Theorem~2(iv)]{PurkaitDecomposition},
\[
\operatorname{Sh}_t(f)=\widehat a_f(t)F.
\]
Combining this identity with the defining Dirichlet series in
\cite[Theorem~2]{PurkaitDecomposition} gives the exact identity
\[
\sum_{n\ge1}\frac{a_f(tn^2)}{n^{s-k/2+1}}
=a_f(t)
\left(\sum_{\ell\ge1}\frac{\chi_t(\ell)\mu(\ell)}
{\ell^{s-k/2+3/2}}\right)
\left(\sum_{m\ge1}\frac{A_F(m)}{m^s}\right).
\tag{3.1}
\]
Comparing the
coefficients of $n^{-s}$ therefore gives
\[
a_f(tn^2)n^{k/2-1}
=a_f(t)\sum_{\ell m=n}\chi_t(\ell)\mu(\ell)
\ell^{k/2-3/2}A_F(m).
\]
The form $F$ is a fixed integral-weight cuspidal Hecke eigenform. After
expressing it in the finite-dimensional oldspace generated by normalized
newforms, Deligne's Ramanujan--Petersson bound for those newforms
\cite{Deligne} gives
\[
A_F(m)\ll_{F,\varepsilon}m^{(k-2)/2+\varepsilon}.
\]
It follows that
\begin{align*}
|a_f(tn^2)|
&\ll_{f,\varepsilon}|a_f(t)|n^{1-k/2}
  \sum_{\ell m=n}\ell^{k/2-3/2}m^{(k-2)/2+\varepsilon}\\
&\ll_{f,\varepsilon}|a_f(t)|n^{-1/2}
  \sum_{m\mid n}m^{1/2+\varepsilon}\\
&\ll_{f,\varepsilon}|a_f(t)|n^{2\varepsilon}.
\end{align*}
Replacing $2\varepsilon$ by $\varepsilon$ proves the lemma.
\end{proof}

Define
\[
M_f(s)=\sum_{\substack{t\ge1\\t\ \mathrm{square\text{-}free}}}\frac{a_f(t)}{t^s}.
\]
Using the square-free sieve,
\[
M_f(s)=\sum_{r\ge1}\mu(r)D_r(s),
\qquad
D_r(s)=\sum_{\substack{m\ge1\\r^2\mid m}}\frac{a_f(m)}{m^s}.
\tag{3.2}
\]

\begin{lemma}\label{lem:right-bound}
For every $\delta>0$ and $\tau\in\R$,
\[
D_r(1+\delta+i\tau)\ll_{f,\delta}r^{-2}.
\]
\end{lemma}

\begin{proof}
Write $m=nr^2$ and let $n_0$ be the square-free part of $n$. By Lemma \ref{lem:square-index}, for every $\eta>0$,
\[
a_f(nr^2)\ll_{f,\eta}|a_f(n_0)|r^{2\eta}\left(\frac n{n_0}\right)^\eta.
\]
The result follows by writing $n=n_0\ell^2$, choosing $0<\eta<\delta$, and applying Cauchy--Schwarz together with the absolute convergence, to the right of $1$, of the Rankin--Selberg series
\[
\sum_{n\ge1}\frac{|a_f(n)|^2}{n^s}.
\]
This is the standard Rankin--Selberg unfolding, recalled in
Proposition~\ref{prop:all-second} below.
\end{proof}

By additive character orthogonality,
\[
D_r(s)=\frac1{r^2}\sum_{u\bmod r^2}
\sum_{m\ge1}\frac{a_f(m)\e(mu/r^2)}{m^s}.
\tag{3.3}
\]
Expanding $f$ in its Fourier series gives
\[
\Lambda\left(f,\frac{u}{r^2},s\right)
=\frac{\Gamma(s')}{(2\pi)^{s'}}
\sum_{m\ge1}\frac{a_f(m)\e(mu/r^2)}{m^s}.
\]
Reducing $u/r^2$ to lowest terms yields
\[
D_r(s)=\frac{(2\pi)^{s'}}{\Gamma(s')}\frac1{r^2}
\sum_{d\mid r^2}\sum_{\substack{u\bmod d\\(u,d)=1}}
\Lambda\left(f,\frac ud,s\right).
\tag{3.4}
\]

\begin{lemma}\label{lem:cusp-uniform}
Let $\mathfrak a$ run over the cusp classes of $\Gamma_0(N)$. For every $\delta>0$,
\[
\Lambda(f_{\mathfrak a},q,1+\delta+i\tau)
\ll_{f,\delta}(1+|\tau|)^{k/4+\delta}\e^{-\pi|\tau|/2},
\]
uniformly in $q\in\Q$ and in $\mathfrak a$.
\end{lemma}

\begin{proof}
The invariant Petersson norm of a cusp form is bounded. Applied at the cusp
$\mathfrak a$, this gives
\[
|f_{\mathfrak a}(x+iy)|^2\ll_{f,\mathfrak a}y^{-k/2}.
\]
Parseval's identity over one period therefore yields
\[
\sum_{n+\kappa_{\mathfrak a}>0}|b_{\mathfrak a}(n)|^2
\exp\left(-\frac{4\pi(n+\kappa_{\mathfrak a})y}
{w_{\mathfrak a}}\right)
\ll_{f,\mathfrak a}y^{-k/2}.
\]
Taking $y=w_{\mathfrak a}/Y$ shows that
\[
\sum_{0<n+\kappa_{\mathfrak a}\le Y}|b_{\mathfrak a}(n)|^2
\ll_{f,\mathfrak a}Y^{k/2}.
\tag{3.5}
\]
There are only finitely many cusp classes, so the implied constant may be
chosen uniformly in $\mathfrak a$.

Inserting the Fourier expansion into the Mellin integral gives
\[
\Lambda(f_{\mathfrak a},q,s)
=\Gamma(s')\left(\frac{w_{\mathfrak a}}{2\pi}\right)^{s'}
\sum_{n+\kappa_{\mathfrak a}>0}
\frac{b_{\mathfrak a}(n)
\e((n+\kappa_{\mathfrak a})q/w_{\mathfrak a})}
{(n+\kappa_{\mathfrak a})^{s'}}.
\]
On the line $\Re(s)=1+\delta$, the real part of $s'$ is
$k/4+1/2+\delta$. Decomposing the last series into dyadic intervals and
using (3.5) together with Cauchy--Schwarz, the contribution of
$n+\kappa_{\mathfrak a}\asymp Y$ is $O_{f,\delta}(Y^{-\delta})$.
Thus the series converges absolutely and uniformly in $q$. Stirling's
formula for $\Gamma(s')$ now gives the stated estimate.
\end{proof}

\begin{lemma}\label{lem:left-bound}
For every sufficiently small $\delta>0$,
\[
D_r(-\delta+i\tau)\ll_{f,\delta}(1+|\tau|)^{1+2\delta}r^{2+5\delta}.
\]
\end{lemma}

\begin{proof}
By Lemma \ref{lem:additive-fe} and Remark \ref{rem:structural-fe}, every additive twist in (3.4) satisfies
\[
\Lambda\left(f,\frac ud,s\right)
=\omega(u,d)d^{1-2s}\Lambda(f_{\mathfrak a},q',1-s),
\qquad |\omega(u,d)|=1.
\]
Taking $s=-\delta+i\tau$ and applying Lemma \ref{lem:cusp-uniform} gives
\[
\left|\Lambda\left(f,\frac ud,-\delta+i\tau\right)\right|
\ll_{f,\delta}d^{1+2\delta}(1+|\tau|)^{k/4+\delta}\e^{-\pi|\tau|/2}.
\]
Using \(\varphi(d)\leq d\), we obtain
\[
\frac1{r^2}\sum_{d\mid r^2}\varphi(d)d^{1+2\delta}
\ll
\frac1{r^2}\sum_{d\mid r^2}d^{2+2\delta}
\ll_\delta r^{2+5\delta}.
\]
Stirling's formula for \(1/\Gamma(s')\) then gives the stated bound.
\end{proof}

\begin{proposition}\label{prop:continuation}
The Dirichlet series $M_f(s)$ has a holomorphic continuation to $\Re(s)>3/4$. It has at most polynomial growth in $|\Im(s)|$ on every closed vertical strip contained in this half-plane.
\end{proposition}

\begin{proof}
Formula~(3.4) continues $D_r(s)$ to an entire function: each additive
Mellin transform is entire, and $1/\Gamma(s')$ is entire. We first verify
the growth condition needed for the Phragm\'en--Lindel\"of principle.
Fix $r$, and let $I\subset\R$ be compact. For each rational number $q$
occurring in the finite sum (3.4), set $y=\e^x$ in the defining integral.
Then
\[
\Lambda(f,q,s)
 =\int_{-\infty}^{\infty}
 f(q+i\e^x)\e^{s'x}\,dx.
\]
The Fourier expansion at $\infty$ gives exponential decay as
$x\to+\infty$, while the Fourier expansion at the cusp represented by
$q$ gives exponential decay after the corresponding scaling as
$x\to-\infty$. The same holds after differentiating with respect to $x$.
Repeated integration by parts therefore gives, for every $A>0$,
\[
\Lambda(f,q,\sigma+i\tau)
 \ll_{f,r,I,A}(1+|\tau|)^{-A},
 \qquad \sigma\in I.
\]
Here the dependence on $r$ is harmless: it is used only to verify the
growth condition for the fixed entire function $D_r(s)$, whereas the
explicit dependence on $r$ is supplied by the boundary estimates below.
By Stirling's formula and (3.4), there is a constant $B_I>0$ such that
\[
D_r(\sigma+i\tau)
 \ll_{f,r,I}(1+|\tau|)^{B_I}\e^{\pi|\tau|/2},
 \qquad \sigma\in I.
\]
Thus $D_r(s)$ has finite exponential type in every fixed vertical strip,
and the Phragm\'en--Lindel\"of principle applies in
$-\delta\le\Re(s)\le1+\delta$. This is the convexity argument used in
\cite[Section~4]{HKKL}.

Let $\sigma\in[-\delta,1+\delta]$ and put
\[
\theta=\frac{\sigma+\delta}{1+2\delta}.
\]
Lemmas~\ref{lem:right-bound} and \ref{lem:left-bound}, followed by
Phragm\'en--Lindel\"of, give
\[
D_r(\sigma+i\tau)
\ll_{f,\delta}
(1+|\tau|)^{(1-\theta)(1+2\delta)}
 r^{(1-\theta)(2+5\delta)-2\theta}.
\]
At $\sigma=3/4+\eta$, the exponent of $r$ is
\[
-1-4\eta+O_\eta(\delta).
\]
After fixing $\eta>0$ and choosing $\delta>0$ sufficiently small, we obtain
\[
D_r\left(\frac34+\eta+i\tau\right)
\ll_{f,\eta}(1+|\tau|)^{A_\eta}r^{-1-\rho_\eta}
\]
for suitable $A_\eta,\rho_\eta>0$. Hence the series in (3.2) converges
normally on compact subsets of $\Re(s)>3/4$. The same estimate, with
$\sigma$ ranging over a closed substrip, gives polynomial growth in
$|\Im(s)|$ on every closed vertical strip contained in this half-plane.
\end{proof}

\section{Rankin--Selberg second moments and sign distribution}\label{sec:second-moment}
If $\Gamma_\infty$ denotes the stabilizer of $\infty$ in $\Gamma_0(N)$, then
for ${\rm Re}(s)>1$ we set
\[
E_\infty(z,s)=
\sum_{\gamma\in\Gamma_\infty\backslash\Gamma_0(N)}
\operatorname{Im}(\gamma z)^s,
\qquad
R_f(s)=\sum_{n\ge1}\frac{|a_f(n)|^2}{n^s}.
\]
The Eisenstein series has a meromorphic continuation and a simple pole at
$s=1$, with
\[
\Res_{s=1}E_\infty(z,s)
=\frac{1}{\vol(\Gamma_0(N)\backslash\Hh)};
\]
see \cite[Chapter~6]{IwaniecSpectral}. After establishing the resulting
second-moment asymptotic, we pass from all indices to square-free indices
by the argument of Lau, Royer and Wu
\cite[Section~3]{LRW}.

\begin{proposition}\label{prop:all-second}
There is a constant $C_f>0$ such that
\[
\sum_{n\le X}|a_f(n)|^2\sim C_fX.
\]
With the present normalization,
\[
C_f=\frac{(4\pi)^{k/2}}{\Gamma(k/2)}
\frac{\langle f,f\rangle}{\vol(\Gamma_0(N)\backslash\Hh)}.
\]
\end{proposition}

\begin{proof}
Unfolding
\[
I_f(s)=\int_{\Gamma_0(N)\backslash\Hh}
y^{k/2}|f(z)|^2E_\infty(z,s)\,\frac{dx\,dy}{y^2}
\]
gives
\[
I_f(s)=\frac{\Gamma(s+k/2-1)}{(4\pi)^{s+k/2-1}}R_f(s).
\]
Taking the residue at $s=1$ and using the formula above gives the stated
value of $C_f$. The spectral theory of Eisenstein series shows that
$I_f(s)$ is holomorphic on $\Re(s)\geq1$ except for its simple pole at
$s=1$; see \cite[Chapter~6]{IwaniecSpectral}. Since the gamma factor in
the unfolding identity is holomorphic and nonzero there,
\[
R_f(s)-\frac{C_f}{s-1}
\]
extends continuously to $\Re(s)\geq1$ and is holomorphic on that boundary
away from $s=1$. The coefficients of $R_f(s)$ are nonnegative, so the
Wiener--Ikehara theorem gives the asymptotic formula; see
\cite[Chapter~II.7]{Tenenbaum}.
\end{proof}

Put
\[
B_f(X)=\sum_{\substack{t\le X\\t\ \mathrm{square\text{-}free}}}|a_f(t)|^2.
\]

\begin{proposition}\label{prop:squarefree-second}
Assume that $f$ has a nonzero cuspidal Shimura lift. Then
\[
B_f(X)\asymp_f X.
\]
\end{proposition}

\begin{proof}
The upper bound follows from Proposition~\ref{prop:all-second}. For the
lower bound, we use a variant of the square-free second-moment transfer in
\cite[Lemma~9]{LRW}. Fix $0<\eta<1/2$. Writing each positive integer
uniquely as $m=tn^2$ with $t$ square-free and applying
Lemma~\ref{lem:square-index}, we obtain
\begin{align*}
\sum_{m\leq X}|a_f(m)|^2
&=\sum_{\substack{t\leq X\\t\ \mathrm{square\text{-}free}}}
  \sum_{n\leq (X/t)^{1/2}}|a_f(tn^2)|^2\\
&\ll_{f,\eta}
  \sum_{\substack{t\leq X\\t\ \mathrm{square\text{-}free}}}
  |a_f(t)|^2\sum_{n\leq (X/t)^{1/2}}n^{2\eta}\\
&\ll_{f,\eta}X^{1/2+\eta}
  \sum_{\substack{t\leq X\\t\ \mathrm{square\text{-}free}}}
  \frac{|a_f(t)|^2}{t^{1/2+\eta}}.
\end{align*}
Together with Proposition~\ref{prop:all-second}, this gives a constant
$c_{f,\eta}>0$ such that
\[
\sum_{\substack{t\leq X\\t\ \mathrm{square\text{-}free}}}
\frac{|a_f(t)|^2}{t^{1/2+\eta}}
\geq c_{f,\eta}X^{1/2-\eta}
\]
for all sufficiently large $X$.
It remains to show that the preceding weighted lower bound is not
concentrated on small square-free indices. Put
\[
S_\eta(X)=
\sum_{\substack{t\leq X\\t\ \mathrm{square\text{-}free}}}
\frac{|a_f(t)|^2}{t^{1/2+\eta}}.
\]
Partial summation and the upper bound $B_f(X)\ll_fX$ give
\begin{align*}
S_\eta(X)
&=\frac{B_f(X)}{X^{1/2+\eta}}
 +\left(\frac12+\eta\right)
  \int_1^X\frac{B_f(u)}{u^{3/2+\eta}}\,du\\
&\leq C_{f,\eta}X^{1/2-\eta}
\end{align*}
for some constant $C_{f,\eta}>0$. Choose a fixed $0<\lambda<1$ such that
\[
C_{f,\eta}\lambda^{1/2-\eta}\leq\frac12c_{f,\eta}.
\]
Then
\[
\sum_{\substack{\lambda X<t\leq X\\t\ \mathrm{square\text{-}free}}}
\frac{|a_f(t)|^2}{t^{1/2+\eta}}
\gg_{f,\eta}X^{1/2-\eta}.
\]
Since $t^{-1/2-\eta}\leq(\lambda X)^{-1/2-\eta}$ on this interval,
\[
X^{1/2-\eta}
\ll_{f,\eta}
(\lambda X)^{-1/2-\eta}\bigl(B_f(X)-B_f(\lambda X)\bigr).
\]
Hence $B_f(X)-B_f(\lambda X)\gg_{f,\eta}X$, and therefore
$B_f(X)\gg_fX$ after fixing $\eta$. This proves the lower bound.
\end{proof}

\begin{proof}[Proof of Theorem \ref{thm:general}]
Under the hypotheses of the theorem, the square-free coefficient bound
of Iwaniec \cite{IwaniecCoeff}, in the general-level form recorded in
\cite[Theorem~1]{Waibel}, applies to the coefficient $\widehat a_f(t)$ in
the unnormalized expansion. For fixed $N$, take $n=t$ and $v=w=1$ in
that theorem. Since a fixed form is a linear combination of an
orthonormal basis, the resulting basis estimate gives
\[
|\widehat a_f(t)|\ll_{f,\delta}
t^{k/4-2/7+\delta}
\]
for square-free $t$. In weight $3/2$, the hypothesis
$f\perp U_{3/2}(N)$ is exactly the condition required there. Since
$\widehat a_f(t)=a_f(t)t^{k/4-1/2}$, this becomes
\[
|a_f(t)|\ll_{f,\delta}t^{3/14+\delta}
\tag{4.1}
\]
for every $\delta>0$.

Choose a nonnegative $W\in C_c^\infty([0,\infty))$ such that
$W(u)=1$ for $0\leq u\leq1$ and $W(u)=0$ for $u\geq2$, and put
\[
\widetilde W(s)=\int_0^\infty W(u)u^{s-1}\,du.
\]
For $c>1$, Mellin inversion gives
\[
S_f(X):=\sum_{\substack{t\ge1\\t\,\mathrm{square\text{-}free}}}
a_f(t)W(t/X)
=\frac{1}{2\pi i}\int_{(c)}M_f(s)\widetilde W(s)X^s\,ds.
\]
The function $\widetilde W(s)$ decreases faster than any power of
$|\Im(s)|$ on vertical lines with positive real part. Proposition
\ref{prop:continuation} therefore permits the line to be moved to
$\Re(s)=3/4+\eta$, without crossing a pole, and gives
\[
S_f(X)\ll_{f,W,\eta}X^{3/4+\eta}.
\tag{4.2}
\]
Set
\[
A_f(X)=\sum_{\substack{t\ge1\\ t\,\mathrm{square\text{-}free}}}|a_f(t)|W(t/X).
\]
By (4.1), Proposition \ref{prop:squarefree-second}, and the fact that
$W(t/X)=1$ for $t\le X$,
\[
A_f(X)
\ge \sum_{\substack{t\le X\\ t\,\mathrm{square\text{-}free}}}|a_f(t)|
\gg_{f,\delta}X^{-3/14-\delta}
\sum_{\substack{t\le X\\ t\,\mathrm{square\text{-}free}}}|a_f(t)|^2
\gg_{f,\delta}X^{11/14-\delta}.
\tag{4.3}
\]
Choose $\delta,\eta>0$ so small that
\[
\frac34+\eta<\frac{11}{14}-\delta.
\]
Then $|S_f(X)|\le A_f(X)/2$ for all sufficiently large $X$.

Let $A_f^\pm(X)$ denote the portions of $A_f(X)$ coming from coefficients of the indicated sign. Since
\[
A_f^+(X)+A_f^-(X)=A_f(X),
\qquad
A_f^+(X)-A_f^-(X)=S_f(X),
\]
we obtain
\[
A_f^\pm(X)\gg_{f,\delta}X^{11/14-\delta}.
\tag{4.4}
\]
On the other hand, the support of $W$ and (4.1) give
\[
A_f^\pm(X)\ll_{f,W,\delta}X^{3/14+\delta}P_f^\pm(2X).
\]
Combining this with (4.4), replacing $2X$ by $X$, and absorbing $2\delta$ into $\varepsilon$, we obtain
\[
P_f^\pm(X)\gg_{f,\varepsilon}X^{4/7-\varepsilon}.
\]
Both signs therefore occur infinitely often. If the ordered sequence of nonzero square-free coefficients changed sign only finitely many times, it would eventually have one sign, a contradiction.
\end{proof}

\section{Application to the congruent-number problem}\label{sec:application}

The generating series of the differences $\Delta(n)$ is the cusp form
\[
\Psi(z):=\frac12\bigl(\Theta_{Q_1}(z)-\Theta_{Q_2}(z)\bigr)
       =\sum_{n\geq1}\frac{\Delta(n)}2q^n.
\]
The construction in \cite[Section~3]{Qin} shows that
$\Psi\in S_{3/2}(\Gamma_0(128))$ has trivial character and is an
eigenform for $T(p^2)$ at every odd prime $p$. Its eigenvalue agrees with
the $T(p)$-eigenvalue of the newform
$\Phi(z)=\sum_{m\geq1}B(m)q^m\in S_2(\Gamma_0(32))$ associated with
$E:y^2=x^3-x$, and its Shimura lift is $\Phi$.

The identity
\[
\Theta_{Q_1}(z)-\Theta_{Q_2}(z)=2G(z)\theta_2(z)
\]
of \cite[Lemma~3.4]{Qin}, with
$G(z)=\sum_{m,n\in\Z}(-1)^{m+n}q^{(4m+1)^2+16n^2}$ and
$\theta_2(z)=\sum_{\ell\in\Z}q^{2\ell^2}$, implies that every exponent
occurring in $G(z)\theta_2(z)$ is odd. Hence $\Delta(n)=0$ whenever $n$
is even. By the coefficient formula in \cite[Theorem~2.1]{PurkaitHecke}, the
operator denoted by $T(4)$ at the bad prime is $U(4)$ in our convention.
Consequently,
$\Psi|T(4)=\sum_{n\geq1}\Delta(4n)q^n/2=0$, so $\Psi$ is an eigenform
for $T(p^2)$ at every prime. The odd-prime eigenvalues identify $\Psi$
with the Shimura component associated with $\Phi$. By
\cite[Corollary~5.2]{PurkaitDecomposition}, this places \(\Psi\) in
the Shimura component \(S_{3/2}(128,\mathbf 1,\Phi)\), which is
contained in \(U_{3/2}(128)^\perp\). Therefore
\[
\Psi\in U_{3/2}(128)^\perp.
\]

With the normalization used in this paper,
$\Psi(z)=\sum_{n\geq1}a_\Psi(n)n^{1/4}q^n$ and
$a_\Psi(n)=\Delta(n)/(2n^{1/4})$. Thus $a_\Psi(n)$ and $\Delta(n)$ have
the same sign whenever they are nonzero.

\begin{proof}[Proof of Theorem~\ref{thm:quadratic}]
By Theorem~\ref{thm:general}, for every $\varepsilon>0$,
\[
\#\{t\leq X:t\ \mathrm{square\text{-}free},\ \pm a_\Psi(t)>0\}
\gg_{\Psi,\varepsilon}X^{4/7-\varepsilon}.
\]
The even coefficients vanish, and the positive factor $2t^{1/4}$ does
not change the sign in the identity $\Delta(t)=2t^{1/4}a_\Psi(t)$. This
gives the two estimates in Theorem~\ref{thm:quadratic}.
\end{proof}

In the present setting, the observation of Rodriguez--Villegas \cite{RodriguezVillegas} may be
made more explicit.  The central-value formula
\eqref{eq:central-value} involves only $\Delta(t)^2$ and therefore does
not distinguish the two inequalities
\[
r_{Q_1}(t)>r_{Q_2}(t)
\qquad\text{and}\qquad
r_{Q_1}(t)<r_{Q_2}(t).
\]
Theorem~\ref{thm:quadratic} shows that each alternative occurs for
$\gg_\varepsilon X^{4/7-\varepsilon}$ odd square-free integers
$t\leq X$.  The proof does not give a further arithmetic interpretation
of the sign of $\Delta(t)$.

For an odd square-free $t$ with $\Delta(t)\ne0$, the central-value formula
\eqref{eq:central-value} gives $L(E_t,1)>0$. Since $E_t$ has complex
multiplication, the contrapositive of the theorem of Coates and Wiles
gives $L(E_t,1)=0$ when $\operatorname{rank}E_t(\Q)>0$
\cite{CoatesWiles}. Consequently every
integer counted in Theorem~\ref{thm:quadratic} is non-congruent.

For a numerical comparison of the two signs, define
\[
N_\pm(X)=\#\{t\leq X:t\text{ is odd and square-free},\
             \ \pm\Delta(t)>0\}
\]
and
\[
N_{\ne0}(X)=N_+(X)+N_-(X).
\]
Thus $N_{\ne0}(X)$ counts the odd square-free integers $t\le X$ for which
$\Delta(t)\ne0$. Direct enumeration of the two representation numbers gives
\begin{table}[ht]
\centering
\small
\setlength{\tabcolsep}{4pt}
\begin{tabular}{rrrrcc}
\toprule
$X$ & $N_+(X)$ & $N_-(X)$ & $N_{\ne0}(X)$
& $N_+(X)/N_{\ne0}(X)$ & $N_-(X)/N_{\ne0}(X)$ \\
\midrule
$10^3$ & $81$ & $84$ & $165$ & $0.490909$ & $0.509091$ \\
$10^4$ & $885$ & $857$ & $1742$ & $0.508037$ & $0.491963$ \\
$10^5$ & $8941$ & $8975$ & $17916$ & $0.499051$ & $0.500949$ \\
$10^6$ & $92053$ & $92201$ & $184254$ & $0.499598$ & $0.500402$ \\
\bottomrule
\end{tabular}
\caption{Signs of $\Delta(t)$ at odd square-free integers with
$\Delta(t)\ne0$.}
\label{tab:numerical-signs}
\end{table}

The table suggests the following equidistribution.

\begin{conjecture}\label{conj:sign-equidistribution}
As $X\to\infty$,
\[
\frac{N_+(X)}{N_{\ne0}(X)}\longrightarrow\frac12,
\qquad
\frac{N_-(X)}{N_{\ne0}(X)}\longrightarrow\frac12.
\]
Equivalently, $N_+(X)\sim N_-(X)$.
\end{conjecture}

\section*{Acknowledgments}
The authors are grateful to  Professor Hourong Qin for helpful discussions on the congruent number problem.

\end{document}